\newif\ifANONYMOUS	
\documentclass[runningheads]{llncs}
\newcommand{\Paragraph}[1]{\medbreak\noindent\textbf{#1.}}
\usepackage{xcolor}
\usepackage[pdftex,colorlinks=true,pdfstartview=FitV,linkcolor=black,citecolor=blue,urlcolor=black]{hyperref}

\usepackage[utf8]{inputenc} 
\usepackage[T1]{fontenc}

\usepackage{url}            
\usepackage{graphicx}       
\graphicspath{{./figures/}} 

\usepackage{amsmath,amsfonts,mathtools,amsfonts,stmaryrd}

\newcommand{\Z}{\mathbb Z}                      
\newcommand{\fp}{\mathbb{F}_p}                  
\newcommand{\fptwo}{\mathbb{F}_{p^2}}           
\newcommand{\infinity}{\mathcal{O}}             
\newcommand{\isep}{\mathrel{{.}\,{.}}\nobreak}  

\newcommand{\Ec}{{E}}   
\newcommand{\Pt}{{P}}   
\newcommand{\Qt}{{Q}}   

\newcommand{\xdbl}{\texttt{xdbl}}   
\newcommand{\xdble}{\texttt{xdble}} 
\newcommand{\xtpl}{\texttt{xtpl}}   
\newcommand{\xtple}{\texttt{xtple}} 

\newcommand{\keygen}{\texttt{keygen}}   
\newcommand{\derive}{\texttt{derive}}   
\newcommand{\Decaps}{\texttt{Decaps}}   

\newcommand{\pkalice}{\textrm{pk}_2}    
\newcommand{\skalice}{\textrm{sk}_2}    
\newcommand{\pkbob}{\textrm{pk}_3}      
\newcommand{\skbob}{\textrm{sk}_3}      

\usepackage{cleveref}       

\usepackage{xspace}
\usepackage{lipsum}         
\usepackage{todonotes}      
\usepackage{multirow}       
\usepackage{multicol}       
\usepackage{threeparttable}
\usepackage{booktabs}

\usepackage{algorithm}
\usepackage{algorithmic}

\usepackage{tikz}
\usetikzlibrary{cd,decorations.pathreplacing}

\title{
    Faulty isogenies: a new kind of leakage
}
\ifANONYMOUS\else
\author{
    Gora Adj \inst{1}
    \and
    Jes\'us-Javier Chi-Dom\'inguez \inst{2}
    \and
    V\'{i}ctor Mateu \inst{2}
    \and
    Francisco Rodr\'{i}guez-Henr\'{i}quez \inst{2}
}
\authorrunning{G. Adj et al.}
\institute{%
    Departament de Matem\`{a}tica, Universitat de Lleida, Spain
    \\
    \email{gora.adj@gmail.com}
    \and 
    Cryptography Research Centre, Technology Innovation Institute, Abu Dhabi, UAE 
    \\
    \email{\{jesus.dominguez,victor.mateu,francisco.rodriguez\}@tii.ae}
}
\fi

\begin{document}
\maketitle

\begin{abstract}
In SIDH and SIKE protocols, public keys are defined over quadratic extensions of prime fields.
We present in this work a projective invariant property characterizing affine Montgomery curves defined over prime
fields.
We then force a secret $3$-isogeny chain to repeatedly pass through a curve defined over a prime field in order to
exploit the new property and inject zeros in the $A$-coefficient of an intermediate curve to successfully recover
the isogeny chain one step at a time.
Our results introduce a new kind of fault attacks applicable to SIDH and SIKE.

\end{abstract}

\section{Introduction}\label{sec:introduction}

In a seminar held in 1997, Couveignes proposed an isogeny-based scheme for mimicking the Diffie-Hellman key exchange
protocol~\cite{Smith18}.
Couveignes notes were later posted in~\cite{Couveignes06}.
The first published isogeny-based cryptographic primitive was by Charles, Lauter and Goren
in~\cite{CharlesLG09}, where they proposed a hash function whose collision resistance was extracted from
the problem of path-finding in supersingular isogeny graphs.
As early as 2006, Rostovtsev and Stolbunov introduced in~\cite{StolbunovIsogenyStar}  isogeny-based cryptographic
schemes (this was followed by Stolbunov in~\cite{Stolbunov10}) as potential candidates for post-quantum cryptography.
In his 2010 paper~\cite{Stolbunov10}, Stolbunov proposed a Diffie-Hellman-like protocol whose security guarantees were
based on the difficulty of finding smooth-degree isogenies between ordinary elliptic curves.

Jao and De Feo \cite{JaoF11} proposed in late 2011 a Diffie-Hellman key-exchange scheme, which has as underlying hard
problem the difficulty of constructing isogenies between supersingular elliptic curves defined over quadratic extension
prime field $\fptwo$.
Within the context of the NIST standardization process~\cite{nist_pqc}, it was proposed in~\cite{sike} an isogeny-based
key exchange protocol named SIKE, which can be seen as an SIDH variant equipped with a key encapsulation mechanism.
SIKE was selected as one of the five third-round alternate KEM candidates of the NIST contest.

As we will see in~\autoref{subsec:si:dh-ke}, SIDH and SIKE public keys include not only the image curve of a secret
isogeny, but also the auxiliary images of the other party's two torsion basis points.
It has been long suspected that this extra information can help SIKE cryptanalysis, but until now no efficient passive
attack has been found for SIKE, in spite of several potentially promising results~\cite{Petit17,QuehenKLMPPS21}.
\footnote{See~\cite[\S7]{Costello21} for a compelling  argument about why is unlikely that these torsion point attacks
will ever dent the security provided  by SIDH and SIKE.}

On the contrary, in 2016, Galbraith, Petit, Shani and Ti~\cite{GalbraithPST16} presented an active attack against
SIDH, which exploits the additional torsion-point information included in Alice's public key.
The attack in~\cite{GalbraithPST16} consists of sending to Bob a tuple of manipulated torsion points that appears to be
Alice's legitimate public key~\footnote{\label{ft:valida} There is no known efficient approach for validating public
keys in SIDH or SIKE.
In fact, key validation is so problematic in SIDH or SIKE that if an effective algorithm for validating public keys
would ever be found, then such procedure could also be used to efficiently recover secret keys from public
keys~\cite{Smith18,GalbraithV18}}.
Then, the attacker observes if her public key manipulation produces (or not) errors in the protocol and
by doing so, starting from the least significant bit, she can guess one bit at a time.
This allows recovery of the secret key after a  linear number of queries with respect to the size of Bob's exponent
$e_3$ (\emph{cf.}~\autoref{sec:background}).

As a countermeasure to their attack, the authors of~\cite{GalbraithPST16} recommended applying a variant of the
Fujisaki-Okamoto transform~\cite{FO99}.
Obviously, this mitigation entailed a significant performance cost and SIKE was submitted to the NIST post-quantum
cryptography standardization program as the combination of SIDH along with the Hofheinz–H\"ovelmanns–Kiltz
transform~\cite{HHK17} (a variant of~\cite{FO99}), for enabling a key encapsulation mechanism.

In 2017, Ti~\cite{DBLP:conf/pqcrypto/Ti17} proposed a fault attack that allows to recover Alice's secret isogeny
$\phi_A$ exploiting the fact that if the image of a random point $S \in E[p+1]$ under $\phi_A$ is revealed, then with
high probability, the attack is successful~\cite[Remark 3]{TassoFMP21}.
To accomplish this, the attacker must succeed injecting a fault \emph{before} Alice starts computing her public key.
Recall that in a normal execution of Alice's key generation procedure, she must compute the images of fixed basis
points of degree three under her secret isogeny $\phi_A$.
The attacker's goal is to try to perturb any of these points by injecting a fault in a timely manner.
Recently, the authors of~\cite{TassoFMP21} reported a real implementation of Ti's attack that obtained a  small
effectiveness (mainly due to the difficulty of producing electromagnetic injections at the right moment).

Since Ti's attack only deals with injecting a fault in the key generation procedure, it works equally well for SIDH and
SIKE.
However, Ti's attack always produces an error in the SIDH/SIKE shared secret computation.
Hence, the recovered secret key can only be useful if the key generation procedure commits the sin of using the same
private key, something that should not happen in any
proper implementation of SIDH or SIKE.\footnote{In~\cite{TassoFMP21} the authors suggest multipartite SIDH key exchange
as a more plausible scenario for Ti's attack.}
Further, there exist a simple and inexpensive countermeasure to prevent Ti's attack, which is checking the order or
correctness of the auxiliary points before to their publication~\cite{DBLP:conf/pqcrypto/Ti17,TassoFMP21,Costello21}.

Generic fault injection and side-channel attacks against the Fujisaki-Okamoto-based key encapsulation
mechanism were presented in~\cite{XagawaIUTH21,UenoXTITH21}. In~\cite[\S 4.3]{UenoXTITH21}, the authors
applied the attack of~\cite{GalbraithPST16} to SIKE. To do so, the authors implemented a plaintext-checking oracle
by means of a side-channel exploitation.
The attacker compares the side-channel leakage output by the SIKE decapsulation block when processing a valid
ciphertext with a modified one.
The combination of the active attack of \cite{GalbraithPST16} with the plaintext-checking oracle, allows to have a full
SIKE key recovery by invoking less than 500 oracle accesses (see~\cite[Table 2]{UenoXTITH21}).

In~\cite{campos21},  Campos,  Kr\"{a}mer and  M\"uller presented a safe-error attack against SIKE and CSIDH protocols.
They launched a real attack against the SIKEp434 instantiation of SIKE implemented on a Cortex-m4 processor.
The attack achieved almost 100\% of full key recovery at the price of five fault injections per single bit for a total
of some $1,090$ injections.
We stress that the attack in~\cite{campos21} does not exploit any potential vulnerability
on the isogeny computations of SIKE.

Recently, De Feo, El Mrabet, Gen\^{e}t Kalu{\dj}erovi\'{c}, Guertechin, Ponti\'{e}, and Tasso showed SIKE is vulnerable
to zero-value attack~\cite{cryptoeprint:2022:054}.

\Paragraph{Our Contributions}
We present a new fault-injection adaptive attack on SIDH, which could apply to SIKE under the same plaintext-checking
oracle as in~\cite{UenoXTITH21}.
Our analysis centers on Bob's $3$-isogeny chain computation and relies on the following:

\begin{itemize}
    \item A new $\fp$-invariant property to characterize Montgomery projective curves defined over $\fp$.
    \item Using this $\fp$-invariant, force Bob's 3-isogeny chain computation to pass through an $\fp$-curve.
\end{itemize}

With these results, our fault-injection adaptive attack flow can be summarized as follows:
\begin{enumerate}
    \item \textbf{Send an altered version of Alice's public key to Bob}.
    By using an appropriate choice of points for Alice's
    public key, the attacker can force Bob to return to the curve $E \colon y^2 = x^3 +6x^2+x$ when processing the
    $i$th 3-isogeny.
    \item \textbf{Inject a fault in the $(i + 1)$th 3-isogeny output, such that the imaginary part of the Montgomery
    $A$-coefficient is zeroed}.
    This action allows us to guess the $(i+1)$th trit of Bob's private key.
\end{enumerate}

As a proof of concept, we additionally provide a C-code implementation that simulates the injection and verifies the
correctness of the analysis.
The code is freely available at \url{https://github.com/FaultyIsogenies/faulty-isogeny-code}\;.\\

The reminder of this paper is organized as follows.
\autoref{sec:background} presents some basic definitions of elliptic curves and their isogenies.
In~\autoref{sec:analysis}, we give a detailed description of the new fault injection attack.
In particular,~\autoref{subsec:projective} and~\autoref{subsec:faulty3} detail the $\fp$-invariant property and the
core of the attack, respectively.
\autoref{sec:experiments:discussion} illustrates an actual proof-of-concept implementation of the attack.
We also discuss in~\autoref{sec:experiments:discussion} potential countermeasures to thwart or mitigate the attack.
Finally, we  draw our concluding remarks  in~\autoref{sec:conclusions}.

\section{Preliminaries}\label{sec:background}

In this section we present some basic mathematical definitions for elliptic curves and isogenies.
These definitions are extended and discussed in more detail in~\cite{Washington08}.
Furthermore, we briefly describe the SIDH key agreement protocol (see \cite{JaoF11,DJP14}).

\subsection{Supersingular elliptic curves and their isogenies}\label{susection:basic}

Let $p>3$ be a prime number, $\fp$ the finite field with $p$ elements, and $\fptwo$ its quadratic extension.
In this paper, we consider only elliptic curves $E$ defined over $\fptwo$ that are supersingular of order
$\#E(\fptwo) = {(p+1)}^2$ and in Montgomery form given by
\begin{align}
    E \colon By^2 = x^3 + Ax^2 + x,
    \label{eq:mont}
\end{align}
for some $A,B \in \fptwo$.
If $A, B \in \fp$, we say that $E$ is defined over $\fp$.

The order $d$ of a point $P \in E(\fptwo)$ is the smallest positive integer such that
\begin{align*}
[d]P &= \underbrace{P + \cdots + P}_\text{$d$ times}  = \infinity.
\end{align*}

The $d$-torsion subgroup, denoted by $E[d]$, is the set of points $\{P \in E(\fptwo)\mid [d]P = \infinity\}$.
If $\gcd(p,d) = 1$, then $E[d]$, as a subgroup of $E$, is isomorphic to $\Z/n\Z \times \Z/n\Z$.
The $j$-invariant of an elliptic curve $E$ in Montgomery form is defined as
\begin{align}
    j(E) = \frac{256{(A^2 - 2)}^3}{A^2-4}.
    \label{eq:jinvariant}
\end{align}
Two elliptic curves have the same $j$-invariant if and only if they are isomorphic over some extension field of $\fp$.
An isogeny $\phi \colon E \to E'$ over $\fptwo$ is a non-zero rational map satisfying $\phi(\infinity) = \infinity$.
Every isogeny is a surjective group homomorphism with finite kernel.
Two elliptic curves $E$ and $E'$ are said to be isogenous over $\fptwo$ if there exists an isogeny
$\phi \colon E \to E'$ defined over $\fptwo$, and this happens if and only if $\#E(\fptwo) = \#E'(\fptwo)$.

Let $\phi$ be an isogeny defined over $\fptwo$, then it can be represented as
\[\phi=(r_1(X), r_2(X) \cdot Y),\]
where $r_1,r_2 \in \fptwo(X)$, and is said to be separable if $r_1'(X) \neq 0$, and inseparable otherwise.
Now, let $r_1(X) = p_1(X)/q_1(X)$, where $p_1,q_1 \in K[X]$ with $\gcd(p_1,q_1)=1$.
Then the degree of $\phi$ is $\mbox{max}(\deg p_1, \deg q_1)$, and $\phi$ is separable means that
$\# \ker \phi = \deg \phi$.
All the isogenies considered in this work will be separable.
For brevity, we say $d$-isogeny for a degree-$d$ isogeny.

The dual of a $d$-isogeny $\phi \colon E \to E' $ is the unique $d$-isogeny $\hat{\phi} \colon E' \to E$ such that
$\hat{\phi} \circ \phi = [d]$ and $\phi \circ \hat{\phi} = [d]$.
If $\phi$ has cyclic kernel $\langle P \rangle$ and $E[d] = \langle P, Q\rangle$, then $\langle \phi(Q) \rangle$
is the kernel of $\hat{\phi}$.

Isogenies of small prime degree $d$ are computed in time $O(d)$ by means of V\'elu's formulas.
For a degree $d^e$, one splits a $d^e$-isogeny as the composition of $e$ small $d$-isogenies.

\subsection{Montgomery x-only point arithmetic}\label{subsec:montgomery-x-only-point-arithmetic}

We introduce here the notation $x(P)$ to refer to the x coordinate of an elliptic curve point $P$.
The Montgomery curve model is especially amenable for performing $x$-only differential point addition as detailed
in~\cite{DBLP:journals/jce/CostelloS18}.

Following the SIDH and SIKE notation, we denote:
\begin{itemize}
    \item $x([2]P) \gets \xdbl(x(P),A)$, and $x([2^{e}]P) \gets \xdble(x(P),A,e)$
    \item $x([3]P) \gets \xtpl(x(P),A)$, and $x([3^{e}]P) \gets \xtple(x(P),A,e)$
\end{itemize}
where $A$ is the constant coefficient of a Montgomery curve $E$.
Given two points $P, Q$ on $E$ and a positive integer $k$, the x-coordinate $x(P + [k]Q)$ is efficiently computed by
means of an $(\log_2{k})$-step Montgomery
three point ladder procedure at a per-step cost of one $\xdbl$, and one differential point addition.

Computing 2-, 3-, and 4-isogenies can be done efficiently by applying V\'elu's formulas~\cite{Velu71}, but there is a
more efficient way.
Indeed, the idea of $x$-only point addition can be extended to $x$-only isogenies, which are computed from the
$x$-coordinates of the kernel points ~\cite{DBLP:conf/asiacrypt/CostelloH17,DBLP:conf/pqcrypto/Renes18}.
Let $K$ be a point of order $d\in \{2,3,4\}$, and let $\phi \colon E \to E'$ be the isogeny with kernel
$\langle K \rangle$.
There are two isogeny built-in functions:
\begin{itemize}
    \item Isogeny construction: $A',\;\text{coeff} \gets \texttt{xisog}(x(K))$
    \item Isogeny evaluation: $x(\phi(Q)) \gets \texttt{xeval}(x(Q),\;\text{coeff})$.
\end{itemize}
The $\texttt{xisog}$ function computes the $A$-coefficient of the codomain curve $E'$, together with some data
$\text{coeff}$ that is needed in the $\texttt{xeval}$ function and related to the kernel points.
In $\texttt{xeval}$, one pushes the point $Q$ through the isogeny.

From now on, when we take points as inputs and outputs of $\texttt{xisog}$ and $\texttt{xeval}$ in order to alleviate
the notation, it should be understood the x-coordinate of these points.
As well, for a curve point $P$, writing $P \in \fp$ means here $x(P) \in \fp$.
For $d$-isogenies, we will have $\texttt{xisog}d()$, $\texttt{xeval}d()$.

\subsection{The SIDH protocol at a glance}\label{subsec:si:dh-ke}

SIDH is a key agreement scheme based on computations of isogenies between supersingular montgomery curves.
The protocol consists of 2 algorithms: $\keygen$ and $\derive$.
It also provides the setup information required by the two parties, Alice and Bob, to run the algorithms:

\begin{itemize}
    \item A prime $p = 2^{e_2}3^{e_3} - 1$ and $\fptwo = \fp[i]/(i^2 + 1)$.
    \item The starting curve $E/ \fptwo: y^2 = x^3 + 6x^2 + x$ with $ \# E(\fptwo) = (2^{e_2}3^{e_3})^2$ and
    $j$-invariant $j(E) = 287496$.
    \item The points $P_A, Q_A, D_A = P_A - Q_A $ of order $2^{e_2}$.
    \item The points $P_B, Q_B, D_B = P_B - Q_B$ of order $3^{e_3}$.
\end{itemize}

The first algorithm $\keygen$ allows both Alice and Bob to generate their respective public and private keys.
Alice:
\begin{enumerate}
    \item Randomly selects her private key $\skalice$ from $ \llbracket 1 \isep 2^{e_2 - 1}-1 \rrbracket$.
    \item Computes $R_A = P_A + [\skalice]Q_A$.
    \item Finds the $2^{e_2}$-isogeny $\phi_A$ generated by $R_A$.
    \item Pushes $P_B$, $Q_B$, and $D_B$ through the isogeny $\phi_A$ to get her public key
    $\pkalice = (\phi_A(P_B), \phi_A(Q_B), \phi_A(D_B))$.
\end{enumerate}

Likewise, Bob:
\begin{enumerate}
    \item Randomly selects his private key $\skbob$ from $ \llbracket 1 \isep 3^{e_3 - 1}-1 \rrbracket$.
    \item Computes $R_B = P_B + [\skbob]Q_B$.
    \item Finds the $3^{e_3}$-isogeny $\phi_B$ generated by $R_B$.
    \item Pushes $P_A$, $Q_A$, and $D_A$ through the isogeny $\phi_B$ to get his public key
    $\pkbob = (\phi_B(P_A), \phi_B(Q_A), \phi_B(D_A))$.
\end{enumerate}

The second algorithm $\derive$ receives, as input, a public key and a private key, and computes, as output, the shared
$j$-invariant.
Alice would run $\derive(\pkbob, \skalice)$ which would perform the following operations:
\begin{enumerate}
    \item Compute $\phi_B(R_A)$ as $P_A' + [\skalice]Q_A'$.
    \item Find $2^{e_2}$-isogeny $\psi_A$ generated by $\phi_B(R_A)$.
    \item Obtain the codomain curve $E_{AB}$ from $\psi_A$.
    \item Compute the $j$-invariant of $E_{AB}$.
\end{enumerate}
Similarly, Bob would run $\derive(\pkalice, \skbob)$ to obtain the same $j$-invariant.
In this case, the algorithm would:
\begin{enumerate}
    \item Compute $\phi_A(R_B)$ as $P_B' + [\skbob]Q_B'$.
    \item Find $3^{e_3}$-isogeny $\psi_B$ generated by $\phi_A(R_B)$.
    \item Obtain the codomain curve $E_{BA}$ from $\psi_B$.
    \item Compute the $j$-invariant of $E_{BA}$.
\end{enumerate}

Alice and Bob have now created a shared secret by computing the $j$-invariant of their respective isomorphic curves
$E_{AB}$ and $E_{BA}$.

\section{New leakage on $\fp$-isogenies}\label{sec:analysis}

Most of the known active attacks on SIDH and SIKE focus on the public elliptic curve points.
For example: following the same notation as~\autoref{subsec:si:dh-ke}, Ti's attack~\cite{DBLP:conf/pqcrypto/Ti17} is
centered on the output points $\phi(P')$,  $\phi(Q')$, and $\phi(D')$ of $\keygen$, where it changes the point
$\phi(D')$ by a random point $\phi(T)$ on the curve $E'$.

In contrast, our attack to SIDH injects faults on the curve coefficients by building public trapdoor instances.
In this section we show how injecting zeros on the $A$-coefficient of Montgomery curves allows deciding whether $A$
belongs to $\fp$ or not.
Besides, that, we show how to build inputs that allow us to trigger these fault injections in order to fully recover
Bob's private key.

\subsection{An $\fp$-invariant on projective representations}\label{subsec:projective}

For efficiency reasons, SIDH performs $3$-isogenies using a projectivized coefficient $A \in \fptwo$ that can be
described as,
$(\alpha \colon \beta) = (\tilde{A} + 2C \colon \tilde{A} - 2C)$ such that, $A=\tilde{A}/C$ for some
$\tilde{A}, C \in \fptwo$.
This representation allows the computations to minimize the number of divisions.
Let us  write $\alpha = a + ib$ and $\beta = c + id$ with $i^2= -1$ and $a,b,c,d \in \fp$.
We want to focus on determining when does the quotient $\tilde{A}/C$ belong to $\fp$.

If we extend the equations we have
\begin{align}
    \begin{split}
        A
        &= \frac{\tilde{A}}{C}
        = \frac{2(\alpha + \beta)}{\alpha - \beta} = \frac{2(a + c) + 2(b + d)i}{(a - c) + (b - d)i} \\
        &= \frac{\big(2(a + c) + 2(b + d)i\big)\big({(a - c) - (b - d)i}\big)}{{(a - c)}^2 + {(b - d)}^2} \\
        &= \frac{\big(2(a^2 - c^2) + 2(b^2 - d^2)\big) +
        \big(2(b + d)(a - c) - 2(a + c)(b - d)\big)i}{{(a - c)}^2 + {(b - d)}^2} \\
        &= \frac{\big(2(a^2 + b^2) - 2(c^2 + d^2)\big) + 4\big(ad - bc\big)i}{{(a - c)}^2 + {(b - d)}^2}
    \end{split}
    \label{eqs:coeffs:three}
\end{align}
From~\autoref{eqs:coeffs:three} we have that $A$ belongs to $\fp$ if and only if $4\big(ad - bc\big) = 0$

A trivial way to have $A$ in $\fp$ would be if $b = d = 0$.
In that case we would obtain the projective A-coefficient $(a \colon c)$.
The following lemma ensures that we can always take $b = d = 0$.

\begin{lemma}\label{lemma:curve}
For any Montgomery curve $E$ with affine $A$-coefficient in $\fp$, if $(a+ib \colon c+id)$, for some $a,b,d,c \in \fp$,
is a projective curve coefficient of $E$, then also is $(a \colon {c})$.
\end{lemma}

\begin{proof}
    From~\autoref{eqs:coeffs:three} we have the formula to compute the affine coefficient $A$ from the projective
    coefficient $(a \colon {c})$ is:
    \[A = \frac{\tilde{A}}{C} = \frac{2a^2 - 2c^2}{{(a - c)}^2}.\]
    The same equation applied to the affine coefficient $A'$ from the projective coefficient $(a+ib\colon c+id)$
    results in:
    \[A' = \frac{\tilde{A'}}{C'} = \frac{2(a^2 + b^2) - 2(c^2 + d^2)}{{(a - c)}^2 + {(b - d)}^2}.\]
    Hence

    \begingroup\makeatletter\def\f@size{9.5}\check@mathfonts
    \def\maketag@@@#1{\hbox{\m@th\normalsize\normalfont#1}}%
    \begin{align*}
        \begin{split}
            A = A'
            &\Longleftrightarrow
            \frac{2a^2 - 2c^2}{{(a - c)}^2} = \frac{2(a^2 + b^2) - 2(c^2 + d^2)}{{(a - c)}^2 + {(b - d)}^2} \\
            &\Longleftrightarrow
            \big(2a^2 - 2c^2\big)\big({(a - c)}^2 + {(b - d)}^2\big) =
            \big(2(a^2 + b^2) - 2(c^2 + d^2)\big){(a - c)}^2 \\
            &\Longleftrightarrow
            \big(2a^2 - 2c^2\big){(b - d)}^2 = \big(2b^2 - 2d^2\big){(a - c)}^2 \\
            &\Longleftrightarrow
            2(a - c)(a + c){(b - d)}^2 = 2(b - d)(b + d){(a - c)}^2 \\
            &\Longleftrightarrow
            (a + c)(b - d) = (b + d)(a - c) \\
            &\Longleftrightarrow
            bc = ad.
        \end{split}
    \end{align*}\endgroup
    \qed
\end{proof}

\begin{remark}\label{remark:point}
Notice, the $\fp$-invariant given by~\autoref{lemma:curve} easily extends to projective points $(X \colon Z)$ such that
$X,Z\in \fptwo$ and $x(P) = X/Z$ for some point $P$ on $E$.
To be more precise,
any projective point $(x_0 \colon {z_0})$ describes the same affine point in $\fp$ as another projective point
$(x_0+ix_1\colon z_0+iz_1)$ for all $x_0,x_1,z_0,z_1 \in \fp$  if and only if $x_0z_1=z_0x_1$.
\end{remark}

\subsection{Faulting 3-isogenies}\label{subsec:faulty3}

At the beginning of a SIDH key agreement phase, Bob receives
from Alice her public key $\pkalice = \{\phi_A(\Pt_B), \phi_A(\Qt_B), \phi_A(\Pt_B - \Qt_B)\}$.
A malicious Alice may send to Bob another set of image points that may help her to [partially] guess Bob's secret key
$\skbob$.
Another possibility is that an active attacker Eve launches a man-in-the-middle attack by intercepting Alice image
points and relaying to Bob a different set of points of her choice.
Either way, according to the SIDH specifications, Bob has
little defense to distinguish legitimate image points from [carefully chosen] fake
ones~\textsuperscript{\ref{ft:valida}}.

Furthermore, our security model assumes that Eve has the ability of injecting faults during the execution
of~\autoref{alg:strategy:eval3}.
Let us assume that Bob is executing the $i+1$ iteration
of~\autoref{alg:strategy:eval3}, and that he is in the process of computing the degree-3 isogeny
$\phi_i \colon E_{i} \to E_{i+1}$.
Then, by carefully timing her attack, Eve can inject zeroes into the registers $b \mapsto 0$ and $d \mapsto 0$ of
$A_{i+1}$ in the $i$-th iteration of~\autoref{alg:strategy:eval3}.
By zeroing the imaginary part of the coefficient $A_{i+1}$, Eve can infer if it belongs to $\fp$ or to
$\fptwo\setminus \fp$.
In fact, the series of point triplings performed in line~\autoref{alg:strategy:xtple} could produce degenerated
outputs, and from these degenerated outputs arise non-supersingular curves, which are easy to detect.
If the supersingularity is preserved, then the attacker can infer that the injection affected one $A$-coefficient
in $\fp$, otherwise, $A_{i+1}$ cannot possibly live in $\fp$ as proven in~\autoref{lemma:curve}.

\begin{algorithm}[!hbt]
    \renewcommand{\algorithmiccomment}[1]{\hfill//\;#1}
    \caption{Strategy evaluation for computing $3^{e_3}$-isogenies with cyclic kernel generated by a order-$3^{e_3}$
        point $R$ (for more details see Algorithms 19-20 of SIKE specifications~\cite{sike})}
    \label{alg:strategy:eval3}
    \begin{algorithmic}[1]
        \REQUIRE{Point $R \in E$ of order-$3^{e_3}$. The $A$-coefficient of $E$,
            and a strategy $S$ consisting of $e_3 - 1$ positive integers}
        \ENSURE{Codomain curve $E/\langle R \rangle$ of the $3^{e_3}$-isogeny with kernel $\langle R \rangle$}
        \STATE $K \gets []$
        \STATE $k \gets 0$
        \FOR{$i=0$ to $e_3 - 2$}
        \WHILE{$R$ is not an order-$3$ point}\label{alg:strategy:while:in3}
        \STATE $K.push(R)$

        \STATE $R \gets \texttt{xtple}(R,A,S_k)$
        \label{alg:strategy:xtple}

        \STATE $k \gets k + 1$
        \ENDWHILE\label{alg:strategy:while:out3}
        \STATE $A,\;\textrm{coeff} \gets \texttt{xisog}3(R)$\label{alg:strategy:xisog3}

        \FOR{$j=0$ to $k$}
        \STATE $K_j \gets \texttt{xeval}3(K_j, \textrm{coeff})$

        \ENDFOR
        \STATE $R \gets K.pop()$

        \ENDFOR
        \STATE $A,\;\textrm{coeff} \gets \texttt{xisog}3(R)$
        \label{alg:strategy:xisog3:last}
        \RETURN $A$
    \end{algorithmic}
\end{algorithm}

In the following lines, we discuss in detail how to take advantage of this simple but powerful observation given the
public information from the NIST candidate SIKE.

The current public SIKE parameters corresponding to the third round of the NIST standardization process have $A=6$ and
$B=1$, implying that the Montgomery coefficient of the initial curve $E$ is purely defined over $\fp$.
In addition, the public order-$3^{e_3}$ points $P$ and $Q$ have $x$-coordinates over $\fp$, but
$x(D)=x(P - Q)$ lives in $\fptwo\setminus \fp$ (see SIDH specifications from~\cite{sike}).

Let Bob's private key $\skbob$ be represented in radix three as,
\begin{align*}
    \skbob = \sum^{e_3}_{i = 0} s_i3^i.
\end{align*}
To recover $s_0$, we focus on the four order-$3$ points on $E\colon y^2 = x^3 + 6x^2 + x$, which can be generated by
linear combinations
of $P_3 = [3^{e_3 - 1}]P$ and $Q_3 = [3^{e_3 -1}]Q$.
In other words, any order-$3$ point is either
\[P_3, \; Q_3, \; P_3 + Q_3,\text{ or } P_3 - Q_3.\]
By construction, $P_3$ and $Q_3$ lie on $\fp$, while $(P_3 + Q_3)$ and $(P_3 - Q_3)$ do not.
Given that the secret point of SIDH and SIKE has the form $P + [\skbob]Q$, it implies that the first 3-isogeny
$\phi_1 \colon E \to E_1$ has as  kernel generator either $P_3$, $(P_3 + Q_3)$ or $(P_3 - Q_3)$.
Since the coefficient of the curve $E_0$ lies in $\fp$, then the coefficient of $E_1$  will also lie in $\fp$ if and
only if $\phi_1$ has an order-$3$ point generator $K$ such that $x(K) \in \fp$.

In~\autoref{tb:instances} we show how given a successful fault in the first iteration of~\autoref{alg:strategy:eval3}
after line~\autoref{alg:strategy:xisog3}, an attacker can obtain the value $s_0$ by just knowing the information about
which points are in $\fp$ and which ones are not.

As a way of illustration, since $P_3, Q_3 \in \fp$ and $(P_3 + Q_3), (P_3 - Q_3) \in \fptwo\setminus \fp$, we have that
$P_3$ is the one and only kernel living in $\fp$.
So, if the injected fault produces no error it immediately follows that $s_0 = 0$.
Otherwise, if the fault produces an error, then it is still unclear whether the secret value of $s_0$ is one or two.
In this case, by repeating this attack but this time sending to Bob the public key
\[\pkalice' = (P + Q, Q, P),\]
the attacker can find out the value of $s_0$ as follows.
Notice that when Bob executes once again~\autoref{alg:strategy:eval3}, depending on
the value of $s_0$, it will produce as kernel of $\phi_1$ either $P + [2]Q = P - Q$ or $P + [3]Q = P$.
The former point belongs to $\fptwo\setminus \fp$, whereas the latter point lies in $\fp$.
Hence, if the injected fault produces no error, it immediately follows that $s_0 = 2$.
Otherwise, $s_0 = 1$.

Besides the case $P_3 \in \fp$ and $(P_3 + Q_3), (P_3 - Q_3) \in\fptwo\setminus \fp$, there are other five possible
combinations for these three order-$3$ points.
The analysis of the other five cases is similar to the one given above (corresponding to the first row
of~\autoref{tb:instances}).
So although the details are omitted here, we report in~\autoref{tb:instances} six cases where it is possible to guess
the value of $s_0$ at the first attempt using the public key $\pkalice = (P, Q, P-Q)$.
\autoref{tb:instances} also shows twelve cases where a second fault is required.
If a second attempt is required, the attacker must send to Bob the public
key $\pkalice' = (P+Q, Q, P)$.

\begin{remark}
    If the trit $s_0$ was uniformly sampled, the attacker will successfully learn it after
    an average of $1 + \frac{2}{3} = \frac{5}{3}$ attempts.
    More concretely, as shown
    in~\autoref{tb:instances}, there are six cases where the attacker can learn the value of $s_0$ at the first attempt,
    whereas there exist twelve cases that force the attacker to perform two attempts to fully guess the value of $s_0$.
\end{remark}

\begin{table}[!htb]
    \centering
    \begin{tabular}{l||r|r|r|c|c|c}
        \toprule
        \multirow{2}{*}{$x(Q_3)$} &
        \multirow{2}{*}{$x(P_3)$} &
        \multirow{2}{*}{$x(P_3 + Q_3)$} &
        \multirow{2}{*}{$x(P_3 - Q_3)$} &
        \multicolumn{3}{c}{\textbf{Number of instances}} \\
        \cline{5-7}
        & & & & trit $s = 0$ & trit $s = 1$ & trit $s = 2$ \\
        \midrule
        \multirow{3}{*}{in $\fp$}
        & in $\fp$ & not in $\fp$ & not in $\fp$ & 1 & 2 & 2\\
        & not in $\fp$ & in $\fp$ & not in $\fp$ & 2 & 1 & 2\\
        & not in $\fp$ & not in $\fp$ & in $\fp$ & 2 & 2 & 1\\
        \hline
        \multirow{3}{*}{not in $\fp$}
        & not in $\fp$ & in $\fp$ & in $\fp$ & 1 & 2 & 2\\
        & in $\fp$ & not in $\fp$ & in $\fp$ & 2 & 1 & 2\\
        & in $\fp$ & in $\fp$ & not in $\fp$ & 2 & 2 & 1\\
        \bottomrule
    \end{tabular}
    \caption{
        Number of instances to guess $s$ from two random order-$3$ points $P_3$ and $Q_3$ on $E \colon y^2=x^3+6x^2+x$
        such that $\langle P_3,Q_3 \rangle = E[3]$.
        The first instance refers to $(P_3, Q_3, P_3 - Q_3)$, while the second instance to
        $ (P_3', Q_3', P_3' - Q_3')= (P_3 + Q_3, Q_3, P_3)$.
    }
    \label{tb:instances}
\end{table}

\subsection{Forcing $ E_i \colon y^2=x^3+6x^2+x$}
\label{subsec:A6}

In this section, we show how to force Bob's strategy evaluation (\autoref{alg:strategy:eval3}) to repeatedly pass at
iteration $i$ through the curve $\Ec \colon y^2~=~x^3 +6x^2 + x$.
This way, an attacker would be able to guess the secret trit $s_i$, by simply determining whether the $(i+1)$th curve
is either in $\fp$ or not.
As illustrated in~\autoref{alg:strategy:eval3}, the $3^{e_3}$-isogeny procedure is split into
$e_3$ small $3$-isogenies, where the only public curves are the first and last ones:

\begin{align*}
    E_0 = E \to E_1 \to &\cdots \to E_{e-1} \to \tilde{E} = E_{e_3}.
\end{align*}

Each $E_i$ has its Montgomery coefficient $A_i\in\fptwo$.
Note that there exist four different kinds of $3$-isogenies with curve domain $E_{i}$, but only two of them map to
either $E_{i+1}$ or $E_{i-1}$ (this last one goes in the reverse direction and, as mentioned
in~\autoref{susection:basic}, corresponds to the dual isogeny).
However, we know the successive image points $\phi_{1} \circ \phi_{2} \circ \cdots \circ \phi_{i} ([3^{e_3-1-i}]Q)$
determine the dual $3$-isogenies $\hat{\phi}_i \colon E_{i} \to E_{i-1}$,
and the kernel of the $(i+1)$th $3$-isogeny is given by

\begin{align*}
[3^{e_3 - 1 - i}](\phi_{1} \circ \phi_{2} \circ \cdots \circ \phi_{i})(P + [\skbob]Q).
\end{align*}

Notice that the trit $s_i$ completely determines the $3$-isogeny
$\phi_i \colon E_{i-1} \to E_{i}$.
In~\autoref{subsec:faulty3} we detailed how to obtain the first trit $s_0$.
In the following lines we show how to fully recover Bob's private key by forcing him to return to curve $E$ when
processing the $i$th isogeny in $\derive$ algorithm.

Once the attacker knows $s_0$, she can move to $E_1$ by using the 3-isogeny $\phi_1 \colon E_0 \to E_1$ with kernel
generator $[3^{e_3-1}](P+[s_0]Q)$.
Next, we look for an order-$3^{e_3}$ point
\[T \in E_1 : [3^{e_3-1}]\phi_1(Q) \neq \pm [3^{e_3-1}]T.\]
Let
\begin{itemize}
    \item $P' = \phi_1(Q) + [s_0]T$,
    \item $Q' = -T$,
    \item $\hat{\phi}_1 \colon E_1 \to E_0$  be the dual 3-isogeny with kernel generator
    \[[3^{e_3-1}](P' + [s_0]Q') = [3^{e_3-1}]\phi_1(Q).\]
\end{itemize}
When Bob runs $\derive$ algorithm with input $\pkalice = (P', Q', P'-Q')$, we know that the second iteration of the
strategy evaluation will pass through $E$.
Consequently, the attacker can guess $s_1$, as she did for $s_0$, by injecting the fault, now in the second iteration,
and using~\autoref{tb:instances}.
In case she needs two instances, she could build the second one with input
\[\pkalice' = (P' + [3]Q', Q', P' + [2]Q').\]

Now that the attacker knows $s_0$ and $s_1$, she can extend this idea to recover the rest of the trits.
As we have shown, the attacker controls the outcome when Bob computes on $E$.
Therefore, to obtain every $s_i$ from $\skbob$ the attacker needs to provide a proper input public key for Bob such
that at iteration $i$ of~\autoref{alg:strategy:eval3} the curve obtained is $E$.
\autoref{fig:backtracking} summarizes this idea.

\begin{figure}[!hbt]
    \centering
    \begin{tikzcd}[every arrow/.append style={dash}]
        E  \arrow[r,"s_0",->]
        \arrow[rrr, start anchor=east, end anchor=west, no head, yshift=1em, decorate, decoration={brace},
            "\text{known part}" above=3pt] &
        E_1\arrow[r,"s_1",->] &
        \cdots\arrow[r,"s_{i-1}",->] &
        E_{i}\arrow[r,"s_{i}",->]\arrow[d,dotted]
        \arrow[rrr, start anchor=east, end anchor=west, no head, yshift=1em, decorate, decoration={brace},
            "\text{unknown part}" above=3pt]&
        E_{i+1}\arrow[r,"s_{i+1}",->] & \cdots\arrow[r,"s_{e_3 - 1}",->] & E_{e_3}\\
        E\arrow[d,dotted]
        \arrow[rrr, start anchor=east, end anchor=west, no head, yshift=1em, decorate, decoration={brace},
            "\text{backtracking}" above=3pt] &
        E_1\arrow[l,"s_{i-1}",->] &
        \cdots\arrow[l,"s_{i-2}",->]&
        E_{i}\arrow[l,"s_{0}",->] &
        & & \\
        E \arrow[r,"\textcolor{red}{s_{i}}",->]
        \arrow[rrr, start anchor=east, end anchor=west, no head, yshift=-1em, decorate, decoration={brace,mirror},
            "\text{fault impact}" below=3pt] &
        F_{i+1}\arrow[r,"s_{i+1}",->] &
        \cdots\arrow[r,"s_{e_3-1}",->] &
        F_{e_3}
    \end{tikzcd}
    \caption{Attack idea: forcing Bob's strategy evaluation to pass through $E \colon y^2 = x^3 +6x^2 + x$
        at the $i$th iteration, and guess $s_i$ by knowing whether the $(i+1)$th curve is defined over $\fp$ or not.}
    \label{fig:backtracking}
\end{figure}
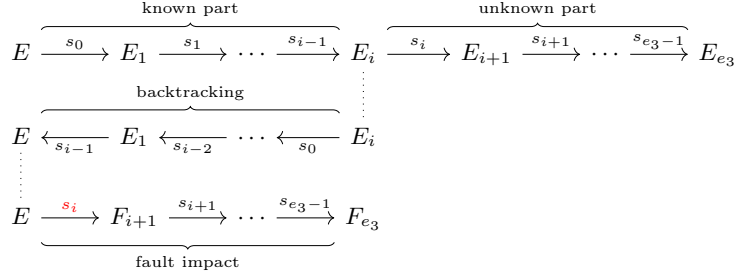

Let us divide the private key $\skbob$ into two parts
\begin{align*}
    \skbob = \sum^{i-1}_{j = 0} s_j3^j + \sum^{e_3}_{t = i} s_t3^t
\end{align*}
and assume the attacker knows $\textrm{sk} = \sum^{i-1}_{j = 0} s_j3^j$.
With this information, she can easily compute the codomain curve $E_i$ of the $3^i$-isogeny $\phi \colon E \to E_i$
with kernel generator
\[[3^{e_3-i}](P+[\mathrm{sk}]Q).\]
The attacker can now forge a public key $\pkalice$ that ensures that the $i$th iteration of Bob's strategy evaluation
passes through $E \colon y^2 = x^3 + 6x^2 + x$ by
\begin{enumerate}
    \item finding an order-$3^{e_3}$ point $T \in E_i : [3^{e_3-1}]T \neq \pm [3^{e_3-1}]\phi(Q)$,
    \item assigning $P' = \phi(Q) + [\mathrm{sk}]T$
    \item assigning $Q' = -T$.
\end{enumerate}
Finally, she  sends
\begin{enumerate}
    \item $\pkalice = (P', Q', P' - Q')$ and,
    \label{pkalice:iteration_i}
    \item $\pkalice' = (P' + [3^i]Q', Q', P' + [3^i-1]Q')$, if needed.
\end{enumerate}

Next, we illustrate how and why the attacker can guess the trit $s_i$ in the following lemma.

\begin{lemma}\label{main:lemma}
Let us assume Bob receives the public key $\pkalice = (P', Q', P' - Q')$ or
$\pkalice' = (P' + [3^i]Q', Q', P' + [3^i-1]Q')$, as previously constructed.
Then, when Bob performs the $\derive$ procedure with the received public key, his $i$th 3-isogenous curve will coincide
with $E \colon y^2 = x^3 + 6x^2 + x$.
\end{lemma}

\begin{proof}
    From the nature of $\pkalice$ and $\pkalice'$, the $i$th 3-isogeny Bob's secret kernel $R$ is
    \[[3^{e_3-i}](P' + [3^i]Q' + [\skbob]Q') = [3^{e_3-i}](P' + [\skbob]Q').\]
    Notice $[3^{e_3-i}]([\skbob]Q') = [3^{e_3-i}]([\mathrm{sk}]Q')$ and thus $R = [3^{e_3-i}]\phi(Q)$.
    By construction, $\phi \colon E \to E_i$ is a $3^i$-isogeny with kernel generator $[3^{e_3-i}](P+[\mathrm{sk}]Q)$.
    That is, $R$ is the kernel generator of the dual $3^i$-isogeny $\hat{\phi} \colon E_i \to E$ since 
    $\langle P + [\skbob]Q, Q \rangle = E\left[3^{e_3}\right]$ holds.
    Consequently, Bob will pass through the $i$th codomain curve $E \colon y^2 = x^3 + 6x^2 + x$.
    On that basis, Bob's $(i+1)$th 3-isogeny will have kernel either $P_3$, $(P_3 + Q_3)$, or $(P_3 + [2]Q_3) = (P_3 - Q_3)$
    as in~\autoref{eq:instances3}, all of them lying on $E \colon y^2 = x^3 + 6x^2 + x$.
    Consequently, our attacker has the possibility of correlating the trit $s_i$ by guessing whether Bob passed through
    $\fp$ 3-isogeny codomain curves, by applying~\autoref{lemma:curve} and using~\autoref{tb:instances}.
    \begin{align}
    \begin{split}
        P_3 & = [3^{e_3 - 1 - i}]\left(\hat{\phi}(P') + [\textrm{sk}]\hat{\phi}(Q')\right), \\
        P_3 + Q_3 & = [3^{e_3 - 1 - i}]\left(\hat{\phi}(P') + [\textrm{sk} + 3^i]\hat{\phi}(Q')\right),\\
        P_3 - Q_3 & = [3^{e_3 - 1 - i}]\left(\hat{\phi}(P') + [\textrm{sk} + 2\cdot3^i]\hat{\phi}(Q')\right).
    \end{split}
    \label{eq:instances3}
    \end{align}
\end{proof}

\subsection{Proposed attack}
\label{subsec:attack3}

Let $\mathbf{O}_{\skbob}(\pkalice, i)$ be an oracle taking as input an SIDH/SIKE public key
$\pkalice = \left(P', Q', P'-Q'\right)$ and an integer $0 \leq i \leq e_3 - 2$.
The oracle internally performs a strategy evaluation with Bob's static private key $\skbob$ and injects a fault at the
$(i+1)$th iteration
(from line~\autoref{alg:strategy:xisog3} of~\autoref{alg:strategy:eval3}) as detailed in~\autoref{subsec:A6}.
Then it outputs $1$ if supersingularity is preserved for $E_{e_3}$, and $0$ otherwise.
Based on~\autoref{subsec:faulty3} and~\autoref{subsec:A6}, and using $\mathbf{O}_{\skbob}(\pkalice, i)$,
\autoref{alg:attack3} correctly reconstructs Bob's private key $\skbob$ with about
$\frac{5e_3}{3} \approx \frac{5\log_2(p)}{6\log_2(3)} \approx 0.53\log_2(p)$ oracle calls.

\begin{algorithm}[!hbt]
    \renewcommand{\algorithmiccomment}[1]{\hfill//\;#1}
    \caption{Bob's private key recovery by injecting zeros on the $A$-coefficients}
    \label{alg:attack3}
    \begin{algorithmic}[1]
        \REQUIRE{Public parameters $P$, $Q$, and $E$ from SIDH. The oracle $\mathbf{O}_{\skbob}()$ as defined
        in~\autoref{subsec:attack3}, and $\tilde{E} = E / \langle P +[\skbob]Q\rangle$}
        \ENSURE{Bob's private key $\skbob = \sum_{i=0}^{e_3-1}s_i3^i$}
        \STATE $\textrm{sk} \gets 0$
        \FOR{$i=0$ to $e_3 - 2$}
        \STATE $\pkalice \gets \left(P', Q', P'-Q'\right)$ as in~\autoref{subsec:A6}
        \STATE Compute $P_3$, $Q_3$, and $(P_3 - Q_3)$ following~\autoref{eq:instances3}
        \STATE $b \gets \mathbf{O}_{\skbob}(\pkalice, i)$
        \label{alg:line:oracle:1st}
        \STATE $c_1 \gets $ number of elements in $\left\{P_3, P_3 + Q_3, P_3 - Q_3\right\}$ lying on $\fp$
        \STATE $c_0 \gets $ number of elements in $\left\{P_3, P_3 + Q_3, P_3 - Q_3\right\}$ not living in $\fp$
        \IF{$c_b = 1$}
        \STATE Guess $s_i$ using~\autoref{tb:instances}
        \STATE $\textrm{sk} \gets \textrm{sk} + s_i3^i$
        \ELSE
        \STATE $\pkalice' \gets \left(P' + [3^i]Q', Q', P' + [3^i - 1]Q'\right)$
        \STATE $b' \gets \mathbf{O}_{\skbob}(\pkalice', i)$
        \label{alg:line:oracle:2nd}
        \IF{$c_{b'} = 1$}
        \STATE Guess $s_i$ using~\autoref{tb:instances}
        \STATE $\textrm{sk} \gets \textrm{sk} + s_i3^i$
        \ELSE
        \STATE Guess $s_i$ by discard from the previous 2 options
        \STATE $\textrm{sk} \gets \textrm{sk} + s_i3^i$
        \ENDIF
        \ENDIF
        \ENDFOR
        \STATE Brute force search on $s_{e_3-1} \in \{0,1,2\}$ such that
        $E / \langle P + [\textrm{sk} + s_{e_3-1}3^{e_3-1}]Q\rangle = \tilde{E}$
        \RETURN $\textrm{sk} + s_{e_3-1}3^{e_3-1}$
    \end{algorithmic}
\end{algorithm}
\footnotetext{$E'$ is the curve determined by $P'$, $Q'$ and $(P'-Q')$.
We set for the case $i=0$, $\psi \colon E \to E$ as the identity map, and $E' = E$. }

\section{Experiments and countermeasures}\label{sec:experiments:discussion}

As a proof-of-concept of our attack, we implemented \autoref{alg:attack3} using the SIDH Library v3.4
(C Edition)\footnote{\url{https://github.com/microsoft/PQCrypto-SIDH}\label{footnote:sidh:lib}}.
Our software implementation simulates the fault injection at lines \autoref{alg:line:oracle:1st}
and \autoref{alg:line:oracle:2nd} of \autoref{alg:attack3} as follows.

The oracle $\mathbf{O}_{\textrm{sk}}(\textrm{pk}, i)$ accesses the curve $E_i \colon y^2 = x^3 +6x^2 + x$, injects
zeros in the $A$-coefficient of $E_{i+1}$, and then performs a supersingularity check by testing  at the $(i+2)$th
iteration $\texttt{xtpl}(R, A) = \infinity$~\footnote{Isogenies preserve the torsion of points, and because $R$ is
the kernel of the next $3$-isogeny, then verifying $R$ is an order-$3$ point determines supersingularity.}.
We ensure that the input public key $\textrm{pk} = (P', Q', P' - Q')$ for $E' \colon y^2 + x^3 + Ax^2 + x$
has been correctly built, by internally verifying $E_i \colon y^2 = x^3 +6x^2 + x$,  using the isogeny
$\psi \colon E' \to E_i$ with kernel generator $[3^{e_3 -i}](P' + [\textrm{sk}_3 \bmod 3^i]Q') = [3^{e_3 - i}]\psi(Q)$
where $E[3^{e_3}] = \langle P, Q \rangle$.

Furthermore, we verify that $P_3$, $P_3 + Q_3$, and $P_3 - Q_3$ are
different 3-order points on $E_i$ for each instance.
Additionally, we set the initial $3^{e_3}$-order points $P$ and $Q$ on $E \colon y^2 = x^3 +6x^2 + x$ as the public
SIDH/SIKE $3^{e_3}$-torsion points.

\autoref{tb:experiments} summarizes the timings and oracle calls simulating the proposed attack,
and reports the average running time of $\frac{5e_3}{3} = \frac{5\log_2(p)}{6\log_2(3)} \approx 0.53\log_2(p)$ oracle
calls.
All of our experiments successfully recover Bob's private key, and were executed on a 2.3 GHz 8-Core
Intel Core i9 machine with 16GB of RAM, using gcc version 11.1.0. In particular, we use the optimized configuration
from the SIDH Library v3.4 (see~\autoref{footnote:sidh:lib}).
Our software implementation is freely available at \url{https://github.com/FaultyIsogenies/faulty-isogeny-code}\;.

\begin{table}[!htb]
    \centering
    \begin{tabular}{l|l|r|r|r}
        \toprule
        Instances &
        $e_3$ &
        Clock Cycles &
        Seconds &
        Oracle calls \\
        \midrule
        SIDHp434 & 137 & 4249 & 4.62 & 226 \\
        SIDHp503 & 159 & 7257 & 7.76 & 263 \\
        SIDHp610 & 192 & 17424 & 18.64 & 318 \\
        SIDHp751 & 239 & 40638 & 42.77 & 396 \\
        \bottomrule
    \end{tabular}
    \caption{
        All measurements are given in millions of clock cycles, and they correspond to the average of 1000 random
        instances
    }
    \label{tb:experiments}
\end{table}

We emphasize that our software C-code implementation is a proof-of-concept simulator that helps to illustrate and
verifying the correctness of the attack.
So, omitting each verification step from the code will reduce the attack latency, giving a faster key recovery.

\Paragraph{Implications to SIKE}
Recently, Ueno et al.~\cite{UenoXTITH21} analyze a plaintext-checking Oracle-based attack;
they focus on the encryption computation from the SIKE decapsulation procedure.
However, our attack targets SIDH derive procedure, but a combination with the plaintext-checking oracle would give a
fault-injection attack on the SIKE decapsulation procedure requiring $\frac{5}{3} + 1$ queries per iteration and,
therefore, $\frac{8e_3}{3} \approx \frac{8\log_2(p)}{6\log_2(3)} \approx 0.84\log_2(p)$ oracle calls.
This extra query comes from the nature of our attack that constructs not valid inputs for decapsulating messages
(as different from attacking SIDH).~\autoref{tb:expected:runtime} summarizes the expected running time for a fully
private key recovery by attacking $\Decaps$ procedure.

\begin{table}[!htb]
    \centering
    \begin{tabular}{l|l|r|r}
        \toprule
        \multirow{2}{*}{Instances} & \multirow{2}{*}{$e_3$} & \multicolumn{2}{c}{Oracle calls} \\ \cline{3-4}
        & & \cite{UenoXTITH21} & This work \\
        \midrule
        SIKEp434 & 137 & 230 & 364 \\
        SIKEp503 & 159 & 266 & 422 \\
        SIKEp610 & 192 & 323 & 512 \\
        SIKEp751 & 239 & 398 & 630 \\
        \bottomrule
    \end{tabular}
    \caption{
        Expected number of plaintext-checking oracle calls and injected faults required to recovering a fixed Bob's
        private key when attacking $\Decaps$.
        In the third column we predict that the required number of oracle calls is just $\frac{5}{3}e_3$, which is
        slightly smaller than the $2e_3$ oracle calls reported in~\cite[Table 2]{UenoXTITH21}.
    }
    \label{tb:expected:runtime}
\end{table}

\Paragraph{Countermeasures}
A pairing test could work as a countermeasure, but the attacker can manipulate the points to bypass this test.
Recall, $P' = \phi(Q) + [\mathrm{sk}]T$ and $Q' = -T$ form a basis, then there is $\mu$ such that
$e(P', Q') = e(P, Q)^\mu$.
So she forces to pass the pairing check by looking for a point $T$ such that $\theta=\mu/2^{e_2}$ is a quadratic residue
and $\sqrt{\theta}$ is invertible modulo $3^{e_3}$.
If that is the case, she sets as basis input $P' \gets [1/\sqrt{\mu}]P'$ and $Q' \gets [1/\sqrt{\mu}]Q'$.
If not, we try with another different $T$.

In~\cite{TassoFMP21}, Tasso \textit{et al.} implemented Ti's attack and additionally proposed a countermeasure to it.
Essentially, they propose to check whether the reconstructed $A$-coefficient from the computed public key
$\pkalice = (P', Q', R)$ equals the last $A$-coefficient in $\keygen$ (output of
line~\autoref{alg:strategy:xisog3:last} of~\autoref{alg:strategy:eval3}).
If there is equality then Bob ensures that $R = (P' - Q')$ and returns $\pkalice$; otherwise, an attack is detected.

Nevertheless, the new attack in this paper and the ones from~\cite{UenoXTITH21} and~\cite{XagawaIUTH21} use valid
public keys such that $R = (P' - Q')$ always holds, and thus our attack is completely immune to Tasso \textit{et al.}'s
countermeasure and any other variant of it.

From the mechanism of our attack, each instance given by~\autoref{eq:instances3} is on an intermediate curve in
Bob's secret 3-isogeny chain.
Also, the attack is not limited to use exclusively the curve $E \colon y^2 = x^3 + 6x^2 + x$, we only require a
Montgomery curve with affine $A$-coefficient in $\fp$.

Naively, to avoid the attack, Bob should check if any of its intermediate secret 3-isogeny computations returns an
$A$-coefficient in $\fp$, if this happens then Bob would reject Alice's public key.
However, such a countermeasure opens the door to another attack:

\begin{enumerate}
    \item The attacker makes a first guess with $s_0 = 0$
    \item Sends the corresponding $\pkalice~$\footnote{$\pkalice$ and $\pkalice'$ are computed by means
    of~\autoref{eq:instances3}\label{footnote:countermeasure}.} to Bob
    \item If Bob rejects $\pkalice$, then she knows her guess was correct;
    \item If Bob accepts $\pkalice$, then she make another guess with $s_0=1$
    \item Sends a different public key $\pkalice'$~\textsuperscript{\ref{footnote:countermeasure}} to Bob
    \item If Bob rejects $\pkalice'$, then she knows $s_0 = 1$, otherwise, $s_0=2$
    \item Repeat the same attack flow to recover $s_1, \ldots, s_{e_3 - 2}$.
\end{enumerate}

Remarkably, the attack outlined above does not rely on fault injections, but only on oracle calls.
This makes it more dramatic than the attack that Bob tries to avoid in the first place.

Since ~\autoref{lemma:curve} characterizes all curves defined over $\fp$ regardless of whether the affine curve
coefficients belong to $\fp$, randomizing the curve coefficients (assuming multiplying by a random $\fptwo$-element) will not help.
The same goes for randomizing points.
In other words, all Bob's computations operate with projective points and curves defined over
$\fptwo\setminus\fp$, and then Bob cannot be possibly aware that he is processing $\fp$-curves.
Therefore, the attack is not affecting the constant-time nature of SIDH/SIKE implementations.

A plausible countermeasure consist of using a commutative diagram by pushing forward Bob's $3^{e_3}$-isogeny $\psi_B : E_A \rightarrow E_{AB}$ through a random isogeny $\rho : E_A \rightarrow E'_A$ of degree not divisible by 3 and obtain an isogeny $\psi'_B : E'_A \rightarrow E'_{AB}$, then push forward the dual of $\rho$ through $\psi'_B$ to finally obtain $\hat{\rho}' : E'_{AB} \rightarrow E_{AB}$.

For instance, Bob can choose $\rho$ to be of degree $2^k$-isogenies, for some sufficiently large $k$.
To be more precise, Bob would 
\begin{enumerate}
    \item sample an arbitrary order-$2^k$ point $R \in E_A$;
    \item construct $\rho : E_A \rightarrow E'_A$ of kernel $R$;
    \item push Alice's public key $\pkalice$ and a point $D$ through $\rho$, $\rho(D)$ generating $\hat{\rho}$;
    \item compute $\psi'_B : E'_A \rightarrow E'_{AB}$ using $\rho(\pkalice)$, push $\rho(D)$ through $\psi'_B$ to get $D'$;
    \item finally, compute the codomain curve $E_{AB}$ of the isogeny $\hat{\rho}'$ of kernel $D'$.
\end{enumerate}

This randomized countermeasure increases the number of queries, $0.53\log_2(p)$, of our attack by a factor of about $2^k$,  suggesting large (enough) values for $k$ since, asymptotically speaking, it remains the same complexity when $k$ is small compared with $e_2$.
Consequently, a secure countermeasure with, for example, $k=\frac{e_2}{2}$ would give a 2x of slowdown of Bob's computations, while the more conservative choice of $k=e_2$ would be 3x slower.

\Paragraph{Discussion about the feasibility of our attack}
As we have seen, the attack presented in this paper requires injecting
faults in a precise step of the SIDH strategy evaluation procedure.
More concretely, the attacker
must inject zeroes in the imaginary part of the coefficient $A_{i+1}$ at
iteration $i + 1$ of line~\autoref{alg:strategy:xisog3} of~\autoref{alg:strategy:eval3}.

But for a real implementation of SIDH and SIKE,
how feasible/realistic
would be to inject such faults?

In~\cite{TassoFMP21} the authors report that electromagnetic injection
is considerably difficult to synchronize with the execution of SIKE.
Because of this difficulty, the real-scenario attack implemented by the authors
achieved a rather modest experimental success rate of just 0.62\%, which is
relatively inefficient compared with the expected theoretical 50\% success
rate attributed to Ti's attack~\cite[Remark 7]{TassoFMP21}.

One advantage that Ti's attack offers, is that the perturbations
do not need to be placed in a precise step of Key Generation,
provided that the faults are injected before the victim starts computing and evaluating isogenies.
Moreover, with high probability, it suffices that the torsion points are perturbed by flipping any single bit of the
points.

Due to the high timing synchronization  required by our attack, its instantiation using  electromagnetic injection
would appear to be just too difficult for being successfully launched in a real implementation of SIDH.

On the other hand, software oriented fault attacks appear to be more
promising for our attack.

For example, the attack presented in~\cite{campos21} was executed on a ChipWhisperer-Lite board equipped
with a 32-bit STM32F303 ARM Cortex-M4 processor, using the
SIKEp434 Cortex-M4 implementation of~\cite{Peter2019}.
Under the assumption that the attacker knows the exact code locations where the
faults must be injected, the authors reports close to 100\% of efficiency in their attack.
The faults were injected by clock glitching, which forced the
ChipWhisperer card to skip an instruction.
In~\cite{XagawaIUTH21}, the authors report a instruction-skipping generic fault
attack that targets all NIST PQC Round 3 KEM candidates including
SIKE. The attack was executed on a ChipWhisperer cw308 UFO base-board, which
permits fault-injection attacks using clock glitching.

It is conceivable that we can adapt the attacks of~\cite{campos21,XagawaIUTH21} to our setting, but in our case we
should rather focus on perturbing the memory locations where the imaginary parts of the coefficient $A$ are stored.
\section{Concluding remarks}\label{sec:conclusions}
In this paper, we presented a new vulnerability based on deciding whether Bob's $3^{e_3}$-isogeny computation passes
through a curve with $A$-coefficient in $\fp$.
This way we propose a new fault-injection attack that

\begin{itemize}
    \item belongs to the same family of~\cite{GalbraithPST16}, known sometimes as
    \emph{reaction attacks}~\cite{QuehenKLMPPS21},
    \item has similar query complexity as the attacks in~\cite{XagawaIUTH21} and~\cite{GalbraithPST16},
    \item it is different to mitigate than the one previous SIDH fault attack, known as Ti's
    attack~\cite{DBLP:conf/pqcrypto/Ti17},
    \item can be combined with the side-channel attack  in~\cite{UenoXTITH21} to achieve full key recovery in SIKE.
\end{itemize}

From the description of the proposed attack, one can easily see that the best scenarios where it applies are when the
individual isogeny degree is as small as possible, namely 2 or 3.
We considered 3 in this work because the computed isogenies in SIDH/SIKE are of degree 3 and (almost always) 4.
We believe our fault-injection attack can be transformed into a side-channel attack using the techniques
of~\cite{UenoXTITH21} and trying to force working with purely $\fp$ values, but this appears to be a formidable task
that we leave here as future work.

Now, when the individual isogenies have higher (prime) degrees, as in B-SIDH, the attack is unlikely to work as
efficiently as 3-isogenies due to the significantly smaller proportion of $\fp$-isogenies (2 $\fp$-isogenies in
average).
However, the number of queries remains a polynomial factor multiplying the isogeny degree.
Since in B-SIDH, for instance, the prime isogeny degrees are bounded (by at most $2^{16}$) so that the isogenies are
efficiently computable, reaching this number of queries should not be a big concern.

To the best of our knowledge, there is no previous attack in which the attacker can inject such
amount of zeros in memory (or registers) as we require our attacker to do.
We consider finding alternative methods to perform this injection in a controlled way as future work.
We would also like to remark that the attack's timing constraints are similar to those in the loop-abort attack
presented by G\'elin and Wesolowski in 2017~\cite{DBLP:conf/pqcrypto/GelinW17}.
In addition, we believe our contribution opens a new path of attacks not considered before.
It might be later implemented with the appropriate tools as it happened previously in the theoretical attack by Ti in
2017, recently implemented in 2021 by Tasso et al~\cite{TassoFMP21}.

\Paragraph{Acknowledgements}
We thank anonymous reviewers for their helpful comments to improve this work.
We also thank Krijn Reijnders and Michael Meyer for suggesting the randomized countermeasure.

 \bibliographystyle{splncs04}
 \bibliography{references}
 \end{document}